\documentclass[12pt,final]{amsart}
 
\usepackage{a4wide}

\usepackage{amssymb}
\usepackage[foot]{amsaddr} 
\usepackage{mathtools}
\mathtoolsset{showonlyrefs}

\usepackage[colorlinks,bookmarks,linkcolor=black,citecolor=black]{hyperref} 
\usepackage{bookmark}

\def\endofClaim{\hfill\scalebox{.6}{$\Box$}}

\let\eps\varepsilon 
\let\rho\varrho

\def\leq{\leqslant}
\def\geq{\geqslant}

\theoremstyle{plain}
\newtheorem*{theorem*}{Theorem}
\newtheorem{theorem}{Theorem}
\newtheorem{lemma}{Lemma}

\newtheorem{claim}{Claim} 
\theoremstyle{definition}

\theoremstyle{remark} 
 
\newcommand{\oldqed}{}

\newenvironment{claimproof}[1][Proof]{
  \renewcommand{\oldqed}{\qedsymbol}
  \renewcommand{\qedsymbol}{\endofClaim}
  \begin{proof}[#1]
}{
  \end{proof}
  \renewcommand{\qedsymbol}{\oldqed}
}

\DeclarePairedDelimiter\abs{\lvert}{\rvert}
\DeclarePairedDelimiter\floor{\lfloor}{\rfloor}

\begin{document}

\title{Multicolour Ramsey numbers of paths and even cycles}

\author{Ewan Davies}
\author{Matthew Jenssen}
\author{Barnaby Roberts}
\address{London School of Economics and Political Science}
\email{\{e.s.davies,m.o.jenssen,b.j.roberts\}@lse.ac.uk}
\date{\today}

\begin{abstract}
We prove new upper bounds on the multicolour Ramsey numbers of paths and even cycles. It is well known that $(k-1)n+o(n)\leq R_k(P_n)\leq R_k(C_n)\leq kn+o(n)$. 
The upper bound was recently improved by S\'ark\"ozy who showed that $R_k(C_n)\leq\left(k-\frac{k}{16k^3+1}\right)n+o(n)$.  Here we show $R_k(C_n) \leq (k-\frac14)n +o(n)$, obtaining the first improvement to the coefficient of the linear term by an absolute constant. 
\end{abstract}

\maketitle
\thispagestyle{empty}

\section{Introduction}
Ramsey theory is one of the central areas of study in combinatorics and a key problem in the field is that of determining the Ramsey numbers of graphs, defined as follows. 
For a graph $G$, the \emph{Ramsey number} $R_k(G)$ is the least integer $N$ such that any colouring of the edges of the complete graph $K_N$ on $N$ vertices with $k$ colours yields a monochromatic copy of $G$. 
The existence of Ramsey numbers is guaranteed by Ramsey's classical result~\cite{Ram}, but in the case $k\geq3$, determining the value of $R_k(G)$ for a given graph $G$ is in most cases difficult. 
There are only a few graphs $G$ for which we know $R_k(G)$ exactly and often one has to settle for bounds on this quantity. 
In this paper we focus on the case where $G$ is the $n$-vertex path $P_n$, and the case where $n$ is even and $G$ is the $n$-vertex cycle $C_n$. 
The two-colour Ramsey number of a path was completely determined by Gerencs\'er and Gy\'arf\'as~\cite{GerGya} who showed that for $n\geq 2$
\[
R_2(P_n)=\left\lfloor\frac{3n-2}{2}\right\rfloor.
\]
For three colours, Faudree and Schelp \cite{FauSCh2} conjectured that 
\[
R_3(P_n)=\begin{cases}
2n-2 & \text { for } n \text{ even}\,,\\
2n-1 & \text { for } n \text{ odd}\,.
\end{cases}
\]
This conjecture was resolved for large $n$ by Gy\'arf\'as, Ruszink\'o, S\'ark\"ozy and Szemer\'edi~\cite{Gya} but for $k\geq4$ much less is known. 
A well-known upper bound $R_k(P_n) \leq kn$ follows easily by observing that any $k$-colouring of the complete graph on $kn$ vertices contains a colour class with at least $(kn-1)\frac{n}{2}$ edges by the pigeonhole principle. A result of Erd\H{o}s and Gallai~\cite{erdHos1959maximal} (Lemma~\ref{lem:EG} below) then implies that any graph on $kn$ vertices with this many edges contains a copy of $P_n$. Despite the simplicity of this observation, the bound was only recently improved upon by S\'ark\"ozy~\cite{Sarkozy} who proved a stability version of Lemma~\ref{lem:EG} and showed that for $k\geq4$ and $n$ sufficiently large, 
\[
R_k(P_n)\leq\left(k-\frac{k}{16k^3+1}\right)n\,.
\]
In this paper we improve on the above result for all $k\geq 4$ reducing the upper bound on $R_k(P_n)$ by an amount that does not deteriorate as $k$ grows.
Our method is similar to that of \cite{Sarkozy} in that we also use results of Erd\H{o}s and Gallai~\cite{erdHos1959maximal}, and Kopylov~\cite{kopylov1977maximal} to bound the number of edges in the densest two colours.
Our improvement comes from using more information about the densest colour in order to obtain stronger bounds on the number of edges in the second densest.

\begin{theorem}\label{thm:RPn}
For $k\geq 4$ and all $n\geq 64k$,
\[
R_k(P_n)\leq \left(k-\frac{1}{4}+\frac{1}{2k}\right)n\,.
\]
\end{theorem}
If $n$ is much larger we can in fact slightly improve on this bound and extend it to even cycles, see Theorem~\ref{thm:RCn} below. 

Since $P_n$ is a subgraph of $C_n$ we have $R_k(P_n)\leq R_k(C_n)$. 
It is believed that for fixed $k$ and even $n$ the Ramsey numbers $R_k(P_n)$ and $R_k(C_n)$ are asymptotically equal. 
This is due to an application of the regularity lemma and the notion of connected matchings pioneered by \L uczak in~\cite{Lucz}. 
Progress on these two problems therefore track each other closely. 
In the case of two colours Faudree and Schelp~\cite{FauSch}, and independently Rosta~\cite{Rosta} showed that $R_2(C_n)=\frac{3n}{2}+1$ for even $n\geq6$. 
For three colours, Benevides and Skokan~\cite{BenSko} proved that $R_3(C_n)=2n$ for sufficiently large even $n$. For $k\geq4$ colours, again very little is known. \L uczak, Simonovits and Skokan \cite{LSS} showed that for $n$ even, $R_k(C_n)\leq kn+o(n)$, and recently
S\'ark\"ozy \cite{Sarkozy} improved this upper bound to $\left(k-\frac{k}{16k^3+1}\right)n+o(n)$.
Here we obtain a strengthening of Theorem~\ref{thm:RPn} for large $n$.

\begin{theorem}\label{thm:RCn}
For $k\geq 4$ and $n$ even
\[
R_k(C_n)\leq \left(k-\frac{1}{4}\right)n + o(n)\,.
\]
\end{theorem}
It is interesting to note that odd cycles behave very differently in this context. Recently the second author and Skokan~\cite{JenSko} showed, via analytic methods, that for $k\geq4$ and $n$ odd and sufficiently large, $R_k(C_n)=2^{k-1}(n-1)+1$. 
This resolved a conjecture of Bondy and Erd\H{o}s \cite{BonErd} for large $n$.

Let us now briefly discuss lower bounds. Constructions based on finite affine planes (see \cite{BieGya}) show that $R_k(P_n) \geq (k-1)(n-1)$, when $k-1$ is a prime power and this lower bound is thought to be closer to the truth than our upper bound.  Yongqi, Yuansheng, Feng, and Bingxi~\cite{YYFB} provide a construction which shows that $R_k(C_n) \geq (k-1)(n-2)+2$ for any $k$ and for even $n$. This construction can easily be modified to give a lower bound on $R_k(P_n)$ for any $k$ and any $n$. We sketch this construction below. 

To see that $R_k(P_n)\geq 2(k-1)\left(\floor*{\frac{n}{2}}-1\right)+1$, consider a complete graph $G$ on vertices $\{0,1,\dotsc, 2k-3\}$ and for $1\leq i\leq k-1$ colour the edges from vertex $i$ to vertices $i+1,\dotsc,i+k-2$ and the edges from vertex $i+k-1$ to vertices $i+k,\dotsc,i+2k-3$ (taken modulo $2k-2$) with colour $c_i$. 
Then each colour $c_1,\dotsc,c_{k-1}$ consists of two vertex-disjoint stars, each on $k-1$ vertices. 
The remaining edges are those of the form $\{j,j+k-1\}$ for $j=0,\dotsc,k-2$ which are coloured with the final colour $c_k$. 
The final colour forms a matching on $k-1$ edges.
Construct $G'$ by `blowing up' each vertex $i$ of $G$ into a set $V_i$ of $\floor*{\frac{n}{2}}-1$ vertices and colour the edges within $V_i$ with colour $c_k$. 
Edges between sets $V_i$ and $V_j$ in $G'$ are coloured with the same colour as the edge $\{i,j\}$ in $G$. 

There is no monochromatic $P_n$ in $G'$ because in colours $c_1,\dotsc,c_{k-1}$, components are bipartite with smallest part size $\floor*{\frac{n}{2}}-1$, hence cannot contain a $P_n$.
The components in colour $c_k$ have less than $n$ vertices and so cannot contain a $P_n$. 
Again, this lower bound is generally considered to be closer to the truth than our upper bound.

\section{Methods}

Throughout the paper, we omit floor and ceiling signs whenever they are not crucial. To prove Theorem~\ref{thm:RPn} we will proceed by contradiction. 
We take a complete graph on $N=(k-\frac{1}{4}+\frac{1}{2k})n$ vertices whose edges have been coloured with $k$ colours and suppose it contains no monochromatic $P_n$. 
First we show that the densest colour has only a few components and these are not too large. 
For the other colours we consider the edges between these components and use the multi-partite structure to bound the number of such edges. 
This gives bound on the total number of edges which is less than $\binom{N}{2}$ which is the desired contradiction. 
The proof is given in Section~\ref{sec:Pnproof}.

The regularity method of \L uczak (see e.g.~\cite{FigLuczArxiv,FigLucz,Lucz}) reduces the problem of finding a monochromatic $C_n$ in a $k$-coloured complete graph to that of finding a monochromatic component containing a matching of $\frac{n}{2}$ edges in a \emph{reduced graph}, which is a $k$-coloured graph missing a small fraction of edges. 
We use the term \emph{connected matching of $\frac{n}{2}$ edges} to mean a connected graph which contains a matching of $\frac{n}{2}$ edges. Using this method we prove Theorem~\ref{thm:RCn} via the following result, which we prove in Section~\ref{sec:Reg}. 

\begin{theorem}\label{thm:connM}
Let $k\geq 4$ be a positive integer, and let $0\leq\delta< \frac{1}{64k^2} $. 
Then for even $n\geq 32k$ and $N=(k-\frac{1}{4})n$ the following holds. Suppose that $G$ is a $k$-coloured, $N$-vertex graph with at least $(1-\delta)\binom{N}{2}$ edges, then we may find a monochromatic connected matching of $\frac{n}{2}$ edges in $G$. 
\end{theorem}

The statement we use to deduce Theorem~\ref{thm:RCn} from Theorem~\ref{thm:connM} is from a paper of Figaj and \L uczak~\cite[Lemma 3]{FigLuczArxiv}.

\begin{lemma}\label{lem:Lucz}
Let $t > 0$ be a real number. If for every $\eps > 0$
there exists $\delta > 0$ and an $n_1$ such that for every even $n > n_1$ and any $k$-coloured graph $G$ with $v(G) > (1 + \eps)tn$ and $e(G) \geq (1 - \delta)\binom{v(G)}{2}$ has a
monochromatic connected matching of $\frac{n}{2}$ edges, then
$R_k(C_n) \leq (t + o(1))n$.
\end{lemma}

Theorem~\ref{thm:RCn} follows from Theorem \ref{thm:connM} by applying Lemma~\ref{lem:Lucz} with $t=k-\frac14$ and for any positive $\eps$ choosing $\delta<\frac{1}{64k^2}$ and $n_1\geq32k$. 

The remainder of this paper is devoted to proving Theorems~\ref{thm:RPn} and~\ref{thm:connM}. We will need the following extremal results for graphs not containing an $n$-vertex path. 

\begin{lemma}[Erd\H{o}s--Gallai \cite{erdHos1959maximal}]\label{lem:EG}
Let $H$ be a graph which does not contain an $n$-vertex path. Then 
\[
e(H)\leq\frac{n-2}{2}v(H)\,.
\]
\end{lemma}

The following simplified version of a result due to Kopylov~\cite{kopylov1977maximal} improves on the above result for  \emph{connected} graphs.

\begin{lemma}\label{lem:connNoPn}
Let $H$ be a connected graph which does not contain an $n$-vertex path. Then 
\[
e(H)\leq\frac{n}{2}\max\left\{n,v(H)-\frac{n}{4}\right\}.
\]
\end{lemma}

Our next result gives a slight improvement of Lemma~\ref{lem:connNoPn} under the additional assumption that $H$ is $c$-partite. 
In this case the bound on $e(H)$ can be improved when $v(H)$ is small.

\begin{lemma}\label{lem:partiteNoPn}
Let $H$ be a $c$-partite connected graph which does not contain an $n$-vertex path. 
Then 
\[
e(H)\leq
\begin{cases}
\left(1-\frac{1}{c}\right)\frac{v(H)^2}{2} & \text{ for } v(H) \leq n\sqrt{\frac{c}{c-1}}\,,\\
\frac{n^2}{2} & \text { for } n\sqrt{\frac{c}{c-1}} < v(H) \leq \frac{5n}{4}\,,\\
\frac{n}{2}\left(v(H)-\frac{n}{4}\right) & \text{ for }  \frac{5n}{4} < v(H)\,.
\end{cases}
\]
\end{lemma}
\begin{proof}
In the `small' case $v(H) \leq n\sqrt{\frac{c}{c-1}}$ we simply use that a $c$-partite graph has at most as many edges as the complete balanced $c$-partite graph on the same number of vertices.
Therefore we conclude that $e(H)\leq \left(1-\frac{1}{c}\right)\frac{v(H)^2}{2}$ without the assumption that $H$ contains no copy of $P_n$. 

The `medium' case where $n\sqrt{\frac{c}{c-1}} < v(H) \leq \frac{5n}{4}$ and the remaining `large' case follow directly from Lemma~\ref{lem:connNoPn} and make no use of the $c$-partite assumption on $H$.
\end{proof}

Lemma~\ref{lem:partiteNoPn} is already strong enough for us to prove Theorem~\ref{thm:RPn}, however we require another modification to prove Theorem~\ref{thm:RCn}.
We will defer its proof to Section~\ref{sec:Reg}.

\begin{lemma}\label{lem:partiteNoMatching}
Let $H$ be a $c$-partite connected graph which does not contain a matching of $\frac{n}{2}$ edges. 
Suppose further that there is a $c$-partition of $H$ such that the sum of the sizes of any two parts is at least $n$.
Then 
\[
e(H)\leq
\begin{cases}
\left(1-\frac{1}{c}\right)\frac{v(H)^2}{2} & \text{ for } v(H) \leq n\sqrt{\frac{c}{c-1}}\,,\\
\frac{n^2}{2} & \text { for } n\sqrt{\frac{c}{c-1}} < v(H) \leq \frac{5n}{4}\,,\\
\frac{n}{2}\left(v(H)-\frac{n}{4}\right) & \text{ for } \frac{5n}{4} < v(H) < \frac{31n}{16}\,,\\
\frac{n}{2}\left(v(H)-\frac{7n}{16}\right) & \text{ for } \frac{31n}{16} \leq v(H) \,.
\end{cases}
\]
\end{lemma}

\section{Paths}\label{sec:Pnproof}

\begin{proof}[Proof of Theorem~\ref{thm:RPn}] 
Let $\alpha = \frac14-\frac{1}{2k}$ and let $G$ be a $k$-coloured complete graph on $N=(k-\alpha)n$ vertices. Let `blue' be one of these colours. 
We proceed by contradiction, supposing that $G$ contains no monochromatic $n$-vertex path. 
Over all such $G$ consider the one in which blue has the most edges. 
In particular $G$ has at least as many blue edges as any other colour. 
The main idea of our argument is to use bounds on the sizes and the number of blue components to bound the number of edges of $G$ which lie inside blue components, and then to bound the number of edges in each other colour that lie between different blue components.

Let $B$ denote the blue subgraph of $G$ and let $B_1,\dotsc, B_c$ be the connected components of $B$. 
Let `red' be the colour that has the most edges lying between blue connected components and let $R'$ denote the $c$-partite graph of those red edges.
We will prove the following two bounds.
Firstly the number of edges (of any colour) within blue components satisfies
\begin{equation}\label{eq:edgesinblue}
\sum_{i=1}^c \binom{v(B_i)}{2} \leq \left(k-2\alpha+5\alpha^2\right)\frac{n^2}{2}\,,
\end{equation}
and secondly the number of red edges between blue components satisfies
\begin{equation}\label{eq:rededgesbetweenblue}
e(R') \leq \left(k-\alpha-\frac{1}{4}\right)\frac{n^2}{2}\,.
\end{equation}
It follows that
\[
e(G) \leq (k-1)e(R')+\sum_{i=1}^c \binom{v(B_i)}{2}
\leq \bigg((k-1)(k-\alpha -\tfrac14) +(k-2\alpha +5\alpha^2)\bigg)\frac{n^2}{2}\,.
\]
Since $e(G)=\binom{N}{2}=(k-\alpha)(k-\alpha-\frac{1}{n})\frac{n^2}{2}$ it is easy to verify that this fails for $\alpha=\frac{1}{4}-\frac{1}{2k}$ and $n\geq 64k$, reaching the desired contradiction.

We now proceed with proving inequalities \eqref{eq:edgesinblue} and \eqref{eq:rededgesbetweenblue}. 
As a first step toward proving \eqref{eq:edgesinblue}, we establish bounds on the size of blue components. 
We first argue that there cannot be large blue components.

\begin{claim}\label{clm:noBlue5n/4}
There is no blue component in $G$ on more than $\frac{5n}{4}$ vertices.
\end{claim}

\begin{claimproof}
For contradiction, suppose there is a blue component $B_1$ on $\beta n$ vertices with $\beta> \frac54$. 
In this case, by Lemma~\ref{lem:connNoPn} we have $e(B_1)\leq(\beta-\frac14)\frac{n^2}{2}$.
Using Lemma~\ref{lem:EG} on the rest of the blue graph, $B$, we obtain
\[
e(B) \leq \left(\beta-\frac{1}{4}\right)\frac{n^2}{2}+\left(k-\alpha-\beta\right)\frac{n^2}{2} = \bigg(k-\alpha - \frac14 \bigg)\frac{n^2}{2}\,.
\]
Since blue is the densest colour we have $e(G)\leq k\cdot e(B)$ and hence
\begin{align*}
\binom{N}{2}=(k-\alpha)\left(k-\alpha-\frac{1}{n}\right)\frac{n^2}{2} &\leq k \left(k-\alpha-\frac{1}{4}\right)\frac{n^2}{2}\\
\alpha^2-k\alpha+\frac{k}{4}-\frac{k-\alpha}{n} &\leq 0\,.
\end{align*}
This fails when $\alpha = \frac14 - \frac{1}{2k}$ and $n \geq 64k$.
\end{claimproof}

The main application of Claim~\ref{clm:noBlue5n/4} is that now when applying Lemma~\ref{lem:connNoPn} to a blue component $B_i$ we obtain the bound $e(B_i)\leq \frac{n^2}{2}$. 
Using this fact we get a tighter bound on the size of blue components. 
Let $x$ be defined by the equation
\[
xn = \sum_{i=1}^c\max\{v(B_i)-n,0\}\,.
\]
We refer to $x$ as the \emph{excess} size of blue components. The motivation for this definition is that we expect blue components to be of size approximately $n$.
\begin{claim}\label{clm:smallx}
We have $x< \alpha$.
\end{claim}
\begin{claimproof}
Let $B_1,\dotsc,B_\ell$ be the blue components with more than $n$ vertices. 
By Lemma~\ref{lem:connNoPn} and Claim~\ref{clm:noBlue5n/4} we have that there are at most $\frac{n^2}{2}$ edges in each of $B_1,\dotsc,B_\ell$.
Using Lemma~\ref{lem:EG} on the rest of the blue graph we have
\[
e(B) \leq \ell\frac{n^2}{2} + (k-\alpha-\ell-x)\frac{n^2}{2} =(k-\alpha - x)\frac{n^2}{2}\,.
\]
Since blue is the densest colour we have $e(G)\leq k\cdot e(B)$, and so
\[
\binom{N}{2}=(k-\alpha)\left(k-\alpha-\frac{1}{n}\right)\frac{n^2}{2} \leq k(k-\alpha-x)\frac{n^2}{2}
\]
therefore
\[
x \leq \alpha -\frac{\alpha^2}{k} + \frac{k-\alpha}{kn}\,,
\]
and in particular $x<\alpha$ for $\alpha=\frac14-\frac{1}{2k}$ and $n\geq 64k$.
\end{claimproof}

With this bound on the excess, we can prove \eqref{eq:edgesinblue}, completing the first part of the proof.

\begin{claimproof}[Proof of inequality~\eqref{eq:edgesinblue}]
By convexity, $\sum_{i=1}^c\binom{v(B_i)}{2}$ is maximised when there is one blue component of size $(1+x)n$ which has all the excess, $(k-2)$ components of size $n$ and one component of size $(1-\alpha-x)n$. Note that this is at least $n/2$ as $x,\alpha < 1/4$.
It follows that
\begin{align*}
\sum_{i=1}^c \binom{v(B_i)}{2} 
&\leq \left((1+x)^2+(k-2)+(1-\alpha-x)^2\right)\frac{n^2}{2}\\
&=\left(k -2\alpha + \alpha^2 + 2\alpha x + 2x^2\right)\frac{n^2}{2}\,.
\end{align*}
Using the bound $x<\alpha$ from Claim~\ref{clm:smallx} completes the argument.
\end{claimproof}

The second step of the proof of Theorem~\ref{thm:RPn} is to bound the number of red edges which lie between different blue components, establishing \eqref{eq:rededgesbetweenblue}. 
We begin with the following claim.
\begin{claim}\label{clm:upperboundc}
The number, $c$, of blue components of $G$ is at most $\frac{4}{3}(k-\alpha)+1$.
\end{claim}

\begin{claimproof}
It suffices to show that all but at most one blue component contain more than $\frac{3n}{4}$ vertices.
Suppose for contradiction that $B_1$ and $B_2$ each have at most $\frac{3n}{4}$ vertices, and let $b$ satisfy $bn=v(B_1\cup B_2)$.
Note that, by the maximality assumption on blue, $B_1\cup B_2$ must contain at least $n-1$ vertices. If not, putting a blue clique on $V(B_1\cup B_2)$ would increase $e(B)$ without creating a blue $P_n$ in $G$.
We therefore have $\frac{n-1}{n} \leq b\leq \frac32$. 
$e(B_1\cup B_2)$ is maximal when both components are cliques and by convexity is maximised when $v(B_1)=\frac{3n}{4}$, $v(B_2)=\left(b-\frac{3}{4}\right)n$. 
Using Lemma~\ref{lem:EG} on the rest of the blue graph we have
\begin{align*}
e(B) 
&\leq \frac{9}{16}\frac{n^2}{2} + \left(b-\frac{3}{4}\right)^2\frac{n^2}{2}
 + (k-\alpha-b)\frac{n^2}{2}\\
&\leq \left(b^2-\frac{5}{2}b+k-\alpha+\frac{9}{8}\right)\frac{n^2}{2}\,.
\intertext{
Under the constraint $\frac{n-1}{n} \leq b\leq \frac32$, the quadratic function $b^2-\frac{5b}{2}$ is maximised at $b=\frac{n-1}{n}$, hence 
}
e(B)
&\leq \left(k-\alpha -\frac{3}{8} + \frac{1}{2n} + \frac{1}{n^2}\right)\frac{n^2}{2}\,.
\end{align*}
Since blue is the densest colour we have $e(G) \leq k\cdot e(B)$ which gives
\[
\binom{N}{2}=(k-\alpha)\left(k-\alpha-\frac{1}{n}\right)\frac{n^2}{2} 
\leq k\left(k-\alpha -\frac{3}{8} + \frac{1}{2n} + \frac{1}{n^2}\right)\frac{n^2}{2}
\]
hence
\[
\alpha^2 -\alpha k+ \frac{3k}{8} -\frac{3k-2\alpha}{2n}-\frac{k}{n^2}\leq 0\,.
\]
However this fails for $\alpha=\frac14-\frac{1}{2k}$ and $n\geq 64k$ giving the desired contradiction.
\end{claimproof}

Using the above bound on $c$, the next claim uses Lemma~\ref{lem:partiteNoPn} to bound the number of edges of $R'$. 

\begin{claim}\label{cl:edgespervertex}
Let $H$ be a $c$-partite connected graph on at most $(k-\alpha)n$ vertices which does not contain an $n$-vertex path. 
Then
\[
\frac{e(H)}{v(H)}\leq \frac{n}{2}\left(1-\frac{1}{4(k-\alpha)}\right)\,.
\]
\end{claim}
\begin{claimproof}
We use Lemma~\ref{lem:partiteNoPn} to break the proof into three cases depending on the size of $H$.
Firstly in the case where $v(H)\leq n\sqrt{\frac{c}{c-1}}$ we have $\frac{e(H)}{v(H)} \leq (1-\frac{1}{c})\frac{v(H)}{2}$.
Since $v(H)\leq n\sqrt{\frac{c}{c-1}}$ this is at most $\frac{n}{2}\sqrt{\frac{c-1}{c}}\leq\frac{n}{2} \left(1-\frac{1}{2c}\right)$.
By Claim~\ref{clm:upperboundc} we know that $c \leq \frac{4}{3}(k-\alpha)+1$.
This gives a bound of 
\[
\frac{e(H)}{v(H)} \leq \frac{n}{2}\left(1- \frac{3}{8(k-\alpha)+6}\right)\leq \frac{n}{2}\left(1-\frac{1}{4(k-\alpha)}\right) \,.
\]

Next suppose $n\sqrt{\frac{c}{c-1}} < v(H) \leq \frac{5n}{4}$.
Then, by Lemma~\ref{lem:partiteNoPn}, we have $\frac{e(H)}{v(H)} \leq \frac{n}{2}\sqrt{\frac{c-1}{c}}$.
As shown in the previous case $\frac{n}{2}\sqrt{\frac{c-1}{c}}$ is at most $ \frac{n}{2}\left(1-\frac{1}{4(k-\alpha)}\right)$.

Finally suppose $v(H) > \frac{5n}{4}$.
Then $\frac{e(H)}{v(H)} \leq \frac{n}{2}\left(1-\frac{n}{4v(H)}\right)$.
This is maximised when $v(H)$ is as large as possible giving
$\frac{e(H)}{v(H)} \leq \frac{n}{2}\left(1-\frac{1}{4(k-\alpha)}\right)$.
\end{claimproof}

We can now deduce \eqref{eq:rededgesbetweenblue} from Claim~\ref{cl:edgespervertex}. 
There will be a connected component of $R'$ with at least as high a density as the overall density of $R'$.
Therefore if $R'$ had more than $\left(k-\alpha -\frac{1}{4}\right)\frac{n^2}{2}$ edges there would be a connected component $H$ satisfying
\[
\frac{e(H)}{v(H)} > \frac{1}{N}\left(k-\alpha -\frac{1}{4}\right)\frac{n^2}{2} = \frac{n}{2}\left(1-\frac{1}{4(k-\alpha)}\right) \,.
\]
This contradicts Claim~\ref{cl:edgespervertex}, completing the proof.
\end{proof}

\section{Even Cycles}\label{sec:Reg}

The proof of Theorem~\ref{thm:connM} closely resembles the arguments of the previous section. 
We make three changes, the first two of which are only minor adjustments. 
We must work with the value $\alpha=\frac14$ instead of the value $\frac14-\frac{1}{2k}$, and we must permit the host graph $G$ to have as few as $(1-\delta)\binom{N}{2}$ edges for some small $\delta>0$ which we choose. 
The more significant change is that we apply Lemma~\ref{lem:partiteNoMatching} instead of Lemma~\ref{lem:partiteNoPn} to bound the number of edges between blue components. 

The reason we are able to improve upon the result of Theorem~\ref{thm:RPn} is that when looking only for a connected matching (rather than a path) we can better deal with large components of the graph $R'$ consisting of red edges between blue components. 
In particular, the tight case of Claim~\ref{cl:edgespervertex} is when $v(H)>\frac{5n}{4}$ where we can do no better than assume $R'$ consists of one large connected component. 
The improvement in this case is given by Lemma~\ref{lem:partiteNoMatching}, where we get a better bound on $e(H)$ when $H$ is a component of $R'$ with at least $\frac{31n}{16}$ vertices. 

\begin{proof}[Proof of Lemma~\ref{lem:partiteNoMatching}]
First note that the three bounds for the range $v(H) < \frac{31n}{16}$ follow directly from Lemma~\ref{lem:partiteNoPn} since, for even $n$, a copy of $P_n$ contains a matching of $\frac{n}{2}$ edges.
We therefore assume $H$ is a $c$-partite, connected graph on at least $\frac{31n}{16}$ vertices, in which the sizes of any two parts sum to at least $n$ and which contains no matching of $\frac{n}{2}$ edges.
We will show that $e(H) \leq \frac{n}{2}(v(H)-\frac{7n}{16})$.

Let $A=\{v\in H: d(v)\geq n\}$ and let $M$ denote a maximal matching in $H':=H\backslash A$. 
We may assume $v(M)<n$ as $H$ does not contain any matching with $n/2$ edges.
We will bound $e(H)$ by first bounding $e(H')$ and then bounding the number of edges incident to $A$.  
First note that 
\begin{equation}\label{greedy}
\abs{A}\leq\frac{n}{2}-\frac{v(M)}{2}\,,
\end{equation}
otherwise we could greedily extend $M$ to a matching of size $\frac{n}{2}$ in $H$ contradicting the assumption of the lemma. For $v\in H'$ let $d^{\ast}(v)$ denote the number of neighbours of $v$ in $H'\backslash M$. 
Now let $\{u,v\}$ be an edge of $M$. 
Note that either $d^{\ast}(u)\leq1$ or $d^{\ast}(v)\leq 1$ else we could replace the edge $\{u,v\}$ with a pair of edges $\{u,x\}$, $\{v,y\}$ to get a larger matching in $H'$. 
Let us denote the edges of $M$ by $\{u_i, v_i\}$ for $i=1,\ldots,\frac{v(M)}{2}$ and assume without loss of generality that  $d^{\ast}(u_i)\leq1$ for all $i$. 
Since each edge of $H'$ is incident to an edge of $M$ by maximality it follows that
\begin{align}\label{e(H')}
\nonumber e(H')&\leq \sum_{i=1}^{\frac{1}{2}v(M)}\left(d(v_i)+d^{\ast}(u_i)\right)+\binom{\frac{1}{2}v(M)}{2}\\
\nonumber&\leq \frac{v(M)}{2}\left(n+\frac{v(M)}{4}\right)\\
&\leq \frac{5}{8}n^2-\frac{3}{2}\abs{A}n+\frac{\abs{A}^2}{2}\,,
\end{align}
where for the second inequality we used that $d(v)<n$ for all $v\in H'$ by the definition of $A$, $d^{\ast}(u_i)\leq 1$ for all $i$ by assumption. For the last inequality we used \eqref{greedy}. 
We now turn our attention to bounding the number of edges incident to $A$.  
Recall that $H$ is $c$-partite and let $t$ denote the size of its smallest part. 
First let us suppose that $t\leq \abs{A}$. Since we assume the sum of any two parts of $H$ is at least $n$, it follows that the second smallest part of $H$ has size at least $n-t$ (note that $t\leq \frac{n}{2}$ by \eqref{greedy}). 
It follows that at most $t$ vertices of $A$ have degree $v(H)-t$ and the rest have degree at most $v(H)-n+t$ so that
\begin{equation*}
\sum_{v\in A}d(v)\leq t(v(H)-t)+ (\abs{A}-t)(v(H)-n+t)\,.
\end{equation*}
Considering the right hand side as a quadratic function in $t$ we see that it is maximised when $t=\frac{n+\abs{A}}{4}$ and so 
\begin{equation}\label{topt}
\sum_{v\in A}d(v)\leq \frac{\abs{A}^2}{8}+\left(v(H)-\frac{3}{4}n\right)\abs{A}+\frac{n^2}{8}\,.
\end{equation}
Since $e(H)\leq e(H')+\sum_{v\in A}d(v)$ it follows by \eqref{e(H')} and \eqref{topt} that
\begin{equation}\label{e(H)}
e(H)\leq \frac{5}{8}\abs{A}^2+\left(v(H)-\frac{9}{4}n\right)\abs{A}+\frac{3}{4}n^2\,.
\end{equation}
We consider the right hand side as a quadratic function in $\abs{A}$ and optimise under the constraint $0\leq\abs{A}\leq\frac{n}{2}$. The maximum must occur at either $\abs{A}=0$ or $\abs{A}=\frac{n}{2}$ and it is simple to check that the latter is the maximiser under the assumption that $v(H)\geq \frac{31}{16}n$. It follows that $e(H)\leq\frac{n}{2}v(H)-\frac{7}{32}n^2$ as claimed. It remains to consider the case where $t\geq \abs{A}$. Recall that the maximum degree of $H$ is at most $v(H)-t$ and so 
\begin{equation*}
\sum_{v\in A}d(v)\leq \abs{A}(v(H)-t)\leq \abs{A}(v(H)-\abs{A})\leq\frac{n}{2}v(H)-\frac{n^2}{4}\,,
\end{equation*}
where for the last inequality we again use the bound $\abs{A}\leq \frac{n}{2}$. The result follows.
\end{proof}

\begin{proof}[Proof of Theorem~\ref{thm:connM}]
Let $\alpha=\frac14$, $0<\delta<\frac{1}{64k^2}$ and let $G$ be a $k$-coloured graph on $N=(k-\alpha)n$ vertices with at least $(1-\delta)\binom{N}{2}$ edges. 
We proceed by contradiction, supposing that $G$ contains no monochromatic connected matching of $\frac{n}{2}$ edges. 
Over all such $G$ consider the one in which blue has the most edges. 

Let $B_1,\dotsc, B_c$ be the blue connected components of $G$, 
suppose that red has the most edges between blue connected components and let $R'$ denote the $c$-partite graph of red edges which lie between blue components.
The method is the same as the previous section. We establish the following two bounds.
\begin{equation}\label{eq:edgesinblue2}
\sum_{i=1}^c \binom{v(B_i)}{2} \leq \left(k-2\alpha+5\alpha^2\right)\frac{n^2}{2} = \left(k-\frac{3}{16}\right)\frac{n^2}{2}\,,
\end{equation}

\begin{equation}\label{eq:rededgesbetweenblue2}
e(R') \leq \left(k-\alpha -\frac{7}{16}\right)\frac{n^2}{2} =\left(k-\frac{11}{16}\right)\frac{n^2}{2}\,.
\end{equation}
We then deduce
\[
e(G) \leq (k-1)e(R')+\sum_{i=1}^c \binom{v(B_i)}{2}
\leq \bigg(k^2-\frac{11}{16}k+\frac{1}{2}\bigg)\frac{n^2}{2}\,.
\]
Since $e(G)\geq(1-\delta)\binom{N}{2}=(1-\delta)(k-\frac{1}{4})(k-\frac{1}{4}-\frac{1}{n})\frac{n^2}{2}$ it is easy to verify that with $\delta<\frac{1}{64k^2}$ and $n\geq 32k$ we reach the desired contradiction.

It remains to prove the inequalities~\eqref{eq:edgesinblue2} and~\eqref{eq:rededgesbetweenblue2}. 
We start by showing Claims~\ref{clm:noBlue5n/4} and~\ref{clm:smallx} have direct analogues here.

\begin{claim}\label{clm:noBlue5n/42}
There is no blue component on more than $\frac{5n}{4}$ vertices.
\end{claim}

\begin{claimproof}
For contradiction, suppose there is a blue component $B_1$ on $\beta n$ vertices with $\beta> \frac{5n}{4}$. 
In this case, by Lemma~\ref{lem:connNoPn} we have $e(B_1)\leq(\beta-\frac{1}{4})\frac{n^2}{2}$.
Using Lemma~\ref{lem:EG} on the rest of the blue graph, $B$, we obtain
\[
e(B) \leq \left(\beta-\frac{1}{4}\right)\frac{n^2}{2}+\left(k-\alpha-\beta\right)\frac{n^2}{2} = \bigg(k-\alpha - \frac14 \bigg)\frac{n^2}{2}\,.
\]
Since blue is the densest colour we have $e(G)\leq k\cdot e(B)$ and hence
\begin{align*}
(1-\delta)\binom{N}{2}=(1-\delta)(k-\alpha)\left(k-\alpha-\frac{1}{n}\right)\frac{n^2}{2} &\leq k \left(k-\alpha-\frac{1}{4}\right)\frac{n^2}{2}\\
(1-\delta)\alpha^2-(1-2\delta)k\alpha+\frac{k}{4}-(1-\delta)\frac{k-\alpha}{n} -\delta k^2 &\leq 0\,.
\end{align*}
This fails with $\alpha = \frac14$, $n \geq 32k$ and $\delta<\frac{1}{64k^2}$.
\end{claimproof}

Using the above claim we get a tighter bound on the size of blue components. 
Let $x$ be the excess size of blue components
\[
xn = \sum_{i=1}^c\max\{v(B_i)-n,0\}
\]
as before.

\begin{claim}\label{clm:smallx2}
We have $x< \alpha$.
\end{claim}
\begin{claimproof}
Let $B_1,\dotsc,B_\ell$ be the blue components with more than $n$ vertices. 
By Lemma~\ref{lem:connNoPn} and Claim~\ref{clm:noBlue5n/42} we have that there are at most $\frac{n^2}{2}$ edges in each of $B_1,\dotsc,B_\ell$.
Using Lemma~\ref{lem:EG} on the rest of the blue graph we have
\[
e(B) \leq \ell\frac{n^2}{2} + (k-\alpha-\ell-x)\frac{n^2}{2} =(k-\alpha - x)\frac{n^2}{2}\,.
\]
Since blue is the densest colour we have $e(G)\leq k\cdot e(B)$, and so
\[
(1-\delta)\binom{N}{2}=(1-\delta)(k-\alpha)\left(k-\alpha-\frac{1}{n}\right)\frac{n^2}{2} \leq k(k-\alpha-x)\frac{n^2}{2}
\]
therefore
\[
x \leq (1-2\delta)\alpha -(1-\delta)\frac{\alpha^2}{k} +\delta k + (1-\delta)\frac{k-\alpha}{kn}\,,
\]
and in particular $x<\alpha$ for $\alpha=\frac14$, $n\geq 32k$ and $\delta<\frac{1}{64k^2}$.
\end{claimproof}
Inequality~\eqref{eq:edgesinblue2} follows from Claim~\ref{clm:smallx2} in the exact same way as inequality~\eqref{eq:edgesinblue} follows from Claim~\ref{clm:smallx}.

We require the same bound as before on the number of blue components, now with $\alpha = \frac14$.

\begin{claim}\label{clm:upperboundc2}
The number, $c$, of blue components of $G$ is at most $\frac{4}{3}(k-\alpha)+1$.
\end{claim}
\begin{claimproof}
The proof follows that of Claim~\ref{clm:upperboundc} replacing $\binom{N}{2}$ with $(1-\delta)\binom{N}{2}$.
With $\alpha=\frac14$, $n\geq 32k$ and $\delta <\frac{1}{64k^2}$ the required contradiction holds.
\end{claimproof}

The final claim is the improved version of Claim~\ref{cl:edgespervertex} which makes use of Lemma~\ref{lem:partiteNoMatching} and the above claim bounding $c$.

\begin{claim}\label{cl:edgespervertex2}
Let $H$ be a $c$-partite connected graph which does not contain a matching of $\frac{n}{2}$ edges. 
Suppose further that there is a $c$-partition of $H$ such that the sum of the sizes of any two parts is at least $n$.
Then 
\[
\frac{e(H)}{v(H)}\leq \frac{n}{2}\left(1-\frac{7}{16(k-\alpha)}\right)\,.
\]
\end{claim}
\begin{claimproof}
We prove this using Lemma~\ref{lem:partiteNoMatching} in the same way that we proved Claim~\ref{cl:edgespervertex} using Lemma~\ref{lem:partiteNoPn}; breaking into cases depending on the size of $H$.

Firstly if $v(H)\leq n\sqrt{\frac{c}{c-1}}$, or if $n\sqrt{\frac{c}{c-1}} < v(H) \leq \frac{5n}{4}$, the argument is identical to that of Claim~\ref{cl:edgespervertex} giving in both cases $\frac{e(H)}{v(H)} \leq \frac{n}{2}\left(1-\frac{3}{4(k-\alpha)+6} \right)\leq \frac{n}{2}\left(1-\frac{7}{16(k-\alpha)}\right)$.

If $\frac{5n}{4} < v(H) < \frac{31n}{16}$ then 
\[
\frac{e(H)}{v(H)}\leq \frac{n}{2}\left(1-\frac{n}{4v(H)}\right)\leq \frac{n}{2}\left(1-\frac{4}{31}\right)\leq \frac{n}{2}\left(1-\frac{7}{16(k-\alpha)}\right)\,.
\]

Finally if $\frac{31n}{16} \leq v(H)$ we have $\frac{e(H)}{v(H)}\leq \frac{n}{2}\left( 1- \frac{7n}{16v(H)}\right)$.
This is maximised when $v(H)$ is as large as possible and so we have
\[
\frac{e(H)}{v(H)} \leq \frac{n}{2}\left(1-\frac{7}{16(k-\alpha)}\right)\,.
\]
\end{claimproof}

We can now deduce inequality~\eqref{eq:rededgesbetweenblue2} from Claim~\ref{cl:edgespervertex2}.
Suppose for contradiction that $e(R') >\left(k-\alpha -\frac{7}{16}\right)\frac{n^2}{2}$.
Then, by the pigeonhole principle, there is a connected component of $H$ with 
\[
\frac{e(H)}{v(H)} > \frac{1}{N}\left(k-\alpha -\frac{7}{16}\right)\frac{n^2}{2} =\frac{n}{2}\left(1-\frac{7}{16(k-\alpha)}\right)\,.
\]
This contradicts Claim~\ref{cl:edgespervertex2}, proving \eqref{eq:rededgesbetweenblue2} and completing the proof of Theorem~\ref{thm:connM}.
\end{proof}

\bibliographystyle{plain}
\bibliography{evencycles}

\end{document}